\documentclass[a4paper]{amsart}
\usepackage{graphicx}
\usepackage{amsmath}
\usepackage{amssymb}
\usepackage{oldgerm}
\usepackage{mathdots}
\usepackage{stmaryrd}
\usepackage{bm}
\usepackage[all]{xy}
\usepackage{color}
\usepackage{enumerate}
\usepackage{fourier}
\usepackage{hyperref}
\usepackage{mathrsfs}
\usepackage{tikz}
\usetikzlibrary{arrows,%
  decorations.markings,%
  matrix,%
  shapes}

\newcommand{\Hom}{\operatorname{Hom}\nolimits}
\newcommand{\End}{\operatorname{End}\nolimits}

\renewcommand{\mod}{\operatorname{mod}\nolimits}
\newcommand{\stmod}{\operatorname{\underline{mod}}\nolimits}

\newcommand{\Ann}{\operatorname{Ann}\nolimits}

\newcommand{\Ext}{\operatorname{Ext}\nolimits}

\newcommand{\Maxspec}{\operatorname{MaxSpec}\nolimits}

\newcommand{\HH}{\operatorname{HH}\nolimits}

\newcommand{\rad}{\operatorname{rad}\nolimits}

\newcommand{\m}{\mathfrak{m}}

\newcommand{\az}{\mathfrak{a}}
\newcommand{\ra}{\operatorname{\mathfrak{r}}\nolimits}

\newcommand{\op}{\operatorname{op}\nolimits}

\newcommand{\V}{\operatorname{V}\nolimits}

\newcommand{\e}{\operatorname{e}\nolimits}

\newcommand{\Char}{\operatorname{char}\nolimits}

\newcommand{\derived}{\operatorname{\bf{D^{b}}}\nolimits}

\newcommand{\C}{\mathscr{C}}

\newcommand{\D}{\mathscr{D}}
\newcommand{\A}{\mathscr{A}}

\newcommand{\calX}{\mathscr{X}}

\newtheorem{theorem}{Theorem}[section]

\newtheorem{corollary}[theorem]{Corollary}

\theoremstyle{definition}

\theoremstyle{definition}

\theoremstyle{definition}
\newtheorem{fact}[theorem]{Fact}

\theoremstyle{remark}

\theoremstyle{definition}

\theoremstyle{definition}

\theoremstyle{definition}
\newtheorem*{setup}{Setup}

\begin{document}

\title{On support varieties and tensor products for finite dimensional algebras}

\author{Petter Andreas Bergh, Mads Hustad Sand{\o}y, {\O}yvind Solberg}

\address{Institutt for matematiske fag \\ NTNU \\ N-7491 Trondheim \\ Norway}
\email{petter.bergh@ntnu.no}
\email{mads.sandoy@ntnu.no}
\email{oyvind.solberg@ntnu.no}

\subjclass[2010]{16D20, 16E40, 16S80, 16T05, 18D10, 18E30, 81R50}

\keywords{Support varieties, tensor products, quantum complete intersections}


\begin{abstract}
It has been asked whether there is a version of the tensor product property for support varieties over finite dimensional algebras defined in terms of Hochschild cohomology. We show that in general no such version can exist. In particular, we show that for certain quantum complete intersections, there are modules and bimodules for which the variety of the tensor product is not even contained in the variety of the one-sided module.
\end{abstract}

\maketitle

\section{Introduction}\label{sec:intro}

In \cite{Carlson1, Carlson2}, Carlson introduced cohomological support varieties for modules over group algebras of finite groups, using the maximal ideal spectrum of the group cohomology ring. These varieties behave well with respect to the typical operations such as directs sums and syzygies. Moreover, they encode important homological information. For example, the dimension of the support variety of a module equals the complexity of the module. In particular, the variety of a module is trivial if and only if the module is projective.

Shortly after these cohomological support varieties were introduced, it was shown in \cite{AvruninScott} that the variety of a tensor product of modules equals the intersection of the varieties of the modules. This property is commonly referred to as the \emph{tensor product property}. As shown in \cite{FriedlanderPevtsova}, it holds also for modules over finite dimensional cocommutative Hopf algebras; for such algebras, there is a theory of support varieties generalizing that for groups. In fact, one can define support varieties over any finite dimensional Hopf algebra, cocommutative or not, using the Hopf algebra cohomology ring. However, it is not known if this cohomology ring is finitely generated in general. What \emph{is} known is that the tensor product property may or may not hold for non-cocommutative Hopf algebras having finitely generated cohomology rings. Namely, as shown in \cite{BensonWitherspoon, PevtsovaWitherspoon, PlavnikWitherspoon}, there are examples of such algebras where the tensor product property holds, and examples where it does not.

Why do we care about the tensor product property? There are several reasons. Not only does it look good; it indicates that the homological behavior of a tensor product is closely related to each of the factors. When the property does not hold, some peculiar things can happen; examples in \cite{BensonWitherspoon} show that the tensor product of two modules in one order can be projective, but non-projective in the other order. Another reason why the tensor product property is of interest is that in many cases, it is connected with the classification of thick subcategories. It is an ingredient in Balmer's classification of thick tensor ideals of tensor triangulated categories (cf.\ \cite{Balmer}), and a necessary consequence of Benson, Iyengar and Krause's stratification approach in \cite{BIK1, BIK2}, as shown in \cite[Theorem 7.3]{BIK1}. In general, one is often in a situation where some triangulated tensor category (where the tensor product is not necessarily symmetric) acts on a triangulated category, and where the latter comes with a theory of support varieties relative to some cohomology ring; this is studied in detail in \cite{BuanKrauseSnashallSolberg}. If the appropriate tensor product property holds, then it is sometimes the case that the thick subcategories are actually tensor ideals.

In \cite{EHSST, SnashallSolberg, Solberg}, a theory of support varieties for arbitrary finite dimensional algebras was developed, using Hochschild cohomology rings. For such an algebra $A$, there is in general no natural tensor product between one-sided modules, as is the case for Hopf algebras. However, one can tensor any left $A$-module with a bimodule, and obtain a new left $A$-module. It has therefore been asked whether some version of the tensor product property holds in this setting. In other words, given a bimodule $B$ and a left $A$-module $M$, is there an equality
$$\V (B \otimes_A M) = \V(B) \cap \V(M)$$
of support varieties? This does not immediately make sense: how should we define the support variety of a bimodule? If we just use the same definition as for one-sided modules, then the support variety of any bimodule which is one-sided projective is trivial. In this case, the variety of the tensor product $A \otimes_A M$ would be $\V(M)$, whereas $\V(A) \cap \V(M)$ would always be trivial. However, as we explain at the end of Section \ref{sec:main}, there are actually several possible meaningful ways of defining a support variety theory for bimodules, using Hochschild cohomology. On the other hand, we show that the tensor product property can \emph{never} hold in general, regardless of which bimodule version of support variety theory we use. In fact, we show in Theorem \ref{thm:main} that when $A$ is a quantum complete intersection of a certain type, then there exists a left $A$-module $M$ and a bimodule $B$ for which
$$\V (B \otimes_{A} M) \nsubseteq \V_H(M)$$
One consequence of the failure of such an inclusion is that in the stable module category and the derived category of $A$-modules, there are thick subcategories that are not tensor ideals.

\section{Support varieties and tensor products}\label{sec:main}

Let us first recall the basics on the theory of support varieties for finite dimensional algebras, using Hochschild cohomology. We only give a very brief overview; for details, we refer the reader to \cite{EHSST, SnashallSolberg, Solberg}. 

Let $k$ be a field and $A$ a finite dimensional $k$-algebra with radical $\ra$. All modules considered will be finitely generated left modules, and we denote the category of such $A$-modules by $\mod A$. A bimodule over $A$ is the same thing as a left module over the enveloping algebra $A^{\e} = A \otimes_k A^{\op}$, and the Hochschild cohomology ring of $A$ is the graded ring 
$$\HH^*(A) = \bigoplus_{n=0}^{\infty} \Ext_{A^{\e}}^n (A,A)$$
with the Yoneda product. This ring is graded-commutative, and so its even part $\HH^{2*}(A)$ is commutative in the ordinary sense. Now let $M$ and $N$ be $A$-modules, and consider the graded vector space
$$\Ext_A^*(M,N) = \bigoplus_{n=0}^{\infty} \Ext_A^n(M,N)$$
The Yoneda product makes this into a graded left module over $\Ext_A^*(N,N)$, and a graded right module over $\Ext_A^*(M,M)$. Since for every $L \in \mod A$ the tensor product $- \otimes_A L$ induces a homomorphism
$$\varphi_L \colon \HH^*(A) \to \Ext_A^*(L,L)$$
of graded rings, we see that $\Ext_A^*(M,N)$ becomes a module over $\HH^*(A)$ in two ways: via the ring homomorphisms $\varphi_N$ and $\varphi_M$. However, the scalar multiplication via these two ring homomorphisms coincide up to a sign. 

Now suppose that $H$ is a graded subalgebra of $\HH^{2*}(A)$. Then for every pair $(M,N)$ of $A$-modules, we can define the support variety $\V_H(M,N)$ using the maximal ideal spectrum of $H$:
$$\V_H (M,N) = \left \{ \m \in \Maxspec H \mid \Ann_H \left ( \Ext^*_{A} (M,N) \right ) \subseteq \m \right \}$$
There are equalities
$$\V_H (M,M) = \V_H (M,A / \ra ) = \V_H (A/ \ra, M)$$
and we define this to be the support variety $\V_H(M)$ of the single module $M$. These support varieties share many of the properties enjoyed by the cohomological support varieties for modules over group rings, in particular when $H$ is noetherian and $\Ext_A^*(M,N)$ is a finitely generated $H$-module for all $M,N \in \mod A$. If this is the case, we say that the algebra $A$ satisfies \textbf{Fg} with respect to $H$. Note that by \cite[Proposition 5.7]{Solberg}, the (even part of the) Hochschild cohomology ring is universal with this property, in the following sense: the algebra $A$ satisfies \textbf{Fg} with respect to some $H \subseteq \HH^{*}(A)$ if and only if $\HH^{*}(A)$ is noetherian and $\Ext_A^*(A/ \ra, A/ \ra)$ is a finitely generated $\HH^{*}(A)$-module.

The finite dimensional algebras we shall study are of a very special form, namely quantum complete intersections. These are quantum commutative analogues of truncated polynomial rings. Let us therefore fix some notation that we shall use throughout.

\begin{setup}
(1) Fix an algebraically closed field $k$, together with two integers $c \ge 2$ and $a \ge 2$. 

(2) Define an integer $\bar{a}$ by 
$$\bar{a} = \left \{ 
\begin{array}{ll}
a & \text{if } \Char k =0 \\
a/ \gcd (a, \Char k ) & \text{if } \Char k >0
\end{array}
\right.$$
and fix a primitive $\bar{a}$th root of unity $q \in k$.

(3) Denote by $A^c_q$ the quantum complete intersection
$$k \langle x_1, \dots, x_c \rangle / \left ( x_1^a, \dots, x_c^a, \{ x_ix_j -qx_jx_i \}_{i<j} \right )$$
\end{setup}

This is a local selfinjective algebra of dimension $a^c$, and by \cite[Theorem 5.5]{BerghOppermann} it satisfies \textbf{Fg} with respect to $\HH^{2*}(A^c_q)$. In \cite{BensonErdmannHolloway}, it was shown that one can actually define rank varieties over this algebra, and that these varieties behave very much like the rank varieties for group algebras. It was then shown in \cite{BerghErdmann} that these rank varieties are isomorphic to the support varieties one obtains by using a suitable polynomial subalgebra of the Hochschild cohomology ring. We now point out some facts about this algebra and its support varieties.

\begin{fact}\label{fact:extalgebra}
(1) By \cite[Theorem 5.3]{BerghOppermann}, the $\Ext$-algebra $\Ext_{A^c_q}^*(k,k)$ of the simple module $k$ admits a presentation
$$k \langle z_1, \dots, z_c, y_1, \dots, y_c \rangle / \az$$
where $\az$ is the ideal generated by the relations
$$\left (
\begin{array}{ll}
z_iz_j - z_jz_i & \text{for all } i,j \\
z_iy_j - y_jz_i & \text{for all } i,j \\
y_iy_j + qy_jy_i & \text{for all } i > j \\
y_i^2 & \text{for all } i \text{ if } a > 2 \\
y_i^2 - z_i & \text{for all } i \text{ if } a = 2
\end{array}
\right )$$
Here, the homological degree of each $y_i$ is one, whereas that of each $z_i$ is two. In particular, the $z_i$ generate a polynomial subalgebra $k[z_1, \dots, z_c]$ over which $\Ext_{A^c_q}^*(k,k)$ is finitely generated as a module.

(2) As explained in \cite[Section 2]{BerghErdmann}, it follows from \cite[Corollary 3.5]{Oppermann} that the image of the ring homomorphism
$$\varphi_k \colon \HH^{2*}(A^c_q) \to \Ext_{A^c_q}^*(k,k)$$
is the whole polynomial subalgebra $k[z_1, \dots, z_c]$. Consequently, there exists a polynomial subalgebra $k[ \eta_1, \dots, \eta_c]$ of $\HH^{2*}(A^c_q)$ with the following properties: each $\eta_i$ is a homogeneous element in $\HH^{2*}(A^c_q)$ of degree two with $\varphi_k( \eta_i ) = z_i$, and $A^c_q$ satisfies \textbf{Fg} with respect to $k[ \eta_1, \dots, \eta_c]$.
\end{fact}

We now prove our main result. It shows that there exists an $A^c_q$-module $M$ and a bimodule $B$ for which the support variety of the tensor product $B \otimes_{A^c_q} M$ is not contained in the support variety of $M$.

\begin{theorem}\label{thm:main}
Let $k[ \eta_1, \dots, \eta_c ]$ be a polynomial subalgebra of $\HH^{2*}( A^c_q )$ as in \emph{Fact \ref{fact:extalgebra}}. Then for every graded subalgebra $H$ of $\HH^*( A^c_q )$ with 
$$k[ \eta_1, \dots, \eta_c ] \subseteq H \subseteq \HH^{2*}( A^c_q )$$
the following hold:

(1) the algebra $H$ is noetherian, and $A^c_q$ satisfies \textbf{Fg} with respect to $H$;

(2) there exists an $A^c_q$-module $M$ and a bimodule $B$ with $\V_H(B \otimes_{A^c_q} M) \nsubseteq \V_H(M)$.
\end{theorem}

\begin{proof}
Let us simplify notation a bit and write $A$ for our algebra $A^c_q$. Since it satisfies \textbf{Fg} with respect to $k[ \eta_1, \dots, \eta_c ]$, it follows from \cite[Proposition 2.4]{EHSST} that the Hochschild cohomology ring $\HH^*( A )$ is finitely generated as a module over $k[ \eta_1, \dots, \eta_c ]$. Note that the assumption in \cite[Proposition 2.4]{EHSST} is that \textbf{Fg} holds with respect to a graded subalgebra of $\HH^*( A )$ whose degree zero part coincides with $\HH^0( A )$, which is the center of $A$. This is not the case for the polynomial subalgebra $k[ \eta_1, \dots, \eta_c ]$, since the center of $A$ is not of dimension one. However, this assumption is not needed in the result. 

Since $\HH^*( A )$ is finitely generated as a module over the noetherian ring $k[ \eta_1, \dots, \eta_c ]$, the same is true for $H$, since this is a $k[ \eta_1, \dots, \eta_c ]$-submodule of $\HH^*( A )$. Then $H$ is noetherian as a ring, since it contains $k[ \eta_1, \dots, \eta_c ]$ as a subring. Moreover, since $\Ext^*_{A} (k,k)$ is finitely generated over $k[ \eta_1, \dots, \eta_c ]$, it must also be finitely generated over the bigger algebra $H$. This proves (1).

To prove (2), we first show that we may without loss of generality assume that $H = k[ \eta_1, \dots, \eta_c ]$. To do this, consider the ring homomorphism 
$$\varphi_k \colon \HH^*( A ) \to \Ext^*_{A} (k,k)$$
By Fact \ref{fact:extalgebra}, the image of $\HH^{2*}( A )$  is the polynomial subalgebra $k[z_1, \dots, z_c]$ of $\Ext^*_{A} (k,k)$, and this is also the image of $k[ \eta_1, \dots, \eta_c ]$; after all, that is how we constructed $k[ \eta_1, \dots, \eta_c ]$ in the first place. Therefore, since $k[ \eta_1, \dots, \eta_c ] \subseteq H \subseteq \HH^{2*}( A )$, we see that the image of $k[ \eta_1, \dots, \eta_c ]$ is the same as that of $H$, namely $k[z_1, \dots, z_c]$. Now take any $A$-module $X$, and consider its support variety $\V_H(X)$, which by definition is the set
$$\left \{ \m \in \Maxspec H \mid \Ann_H \left ( \Ext^*_{A} (X,X) \right ) \subseteq \m \right \}$$
By \cite[Theorem 3.2]{SnashallSolberg}, there is an equality
$$\V_H (X) = \left \{ \m \in \Maxspec H \mid \Ann_H \left ( \Ext^*_{A} (X,k) \right ) \subseteq \m \right \}$$
and so by \cite[Proposition 3.6]{BerghSolberg} the variety $\V_H(X)$ is isomorphic to the set of maximal ideals of $k[z_1, \dots, z_c]$ containing the annihilator of $\Ext^*_{A} (X,k)$. Here we view $\Ext^*_{A} (X,k)$ as a left module over $\Ext^*_{A} (k,k)$, and in this way it becomes a module over the subalgebra $k[z_1, \dots, z_c]$. The isomorphism respects inclusions of varieties, and this proves the claim.

In light of the above, we now take $H = k[ \eta_1, \dots, \eta_c ]$. Since $k$ is algebraically closed, we may identify the maximal ideal spectrum of $H$ with the affine space $k^c$. For a point $\lambda = ( \lambda_1, \dots, \lambda_c)$ in $k^c$, we denote the corresponding maximal ideal $( \eta_1 - \lambda_1, \dots, \eta_c - \lambda_c)$ in $H$ by $\m_{\lambda}$, and when $\lambda$ is nonzero we denote the corresponding line 
$$\left \{ \left ( \gamma \lambda_1, \dots, \gamma \lambda_c \right ) \mid \gamma \in k \right \}$$
in $k^c$ by $\ell_{\lambda}$. Moreover, we denote the element $\sum_{i=1}^c \lambda_ix_i$ in $A$ by $u_{\lambda}$, and by $F( \lambda )$ the point $( \lambda_1^a, \dots, \lambda_c^a)$ in $k^c$. By \cite[Proposition 3.5]{BerghErdmann}, the support variety $\V_H( Au_{\lambda} )$ of the cyclic $A$-module $Au_{\lambda}$ equals $\ell_{F(\lambda)}$, that is, there is an equality
$$\V_H \left ( Au_{\lambda} \right ) = \left \{ \m_{\gamma F(\lambda)} \mid \gamma \in k \right \} = \left\{  \left ( \eta_1 - \gamma \lambda_1^a, \dots, \eta_c - \gamma \lambda_c^a \right ) \mid \gamma \in k \right \}$$
Note that $F( \lambda ) =0$ if and only if $\lambda = 0$.

Now take any point $\mu = ( \mu_1, \dots, \mu_c )$ in $k^c$ with $\mu_i \neq 0$ for all $i$, and consider the automorphism $\psi_{\mu} \colon A \to A$ given by $x_i \mapsto \mu_i x_i$. What happens to the cyclic $A$-module $Au_{\lambda}$ when we twist it by this automorphism? In general, for an $A$-module $X$ and an automorphism $\psi$ of $A$, the twisted module ${_{\psi}X}$ is the same as $X$ as a vector space, but for $w \in A$ and $x \in X$ the scalar multiplication is $w \cdot x = \psi (w) x$. Now denote the point $( \mu_1^{-1} \lambda_1, \dots, \mu_c^{-1} \lambda_c )$ in $k^c$ by $\mu^{-1} \lambda$, and consider the map
\begin{eqnarray*}
Au_{\mu^{-1} \lambda} & \to & {_{\psi_{\mu}}\left ( Au_{\lambda} \right )} \\
w u_{\mu^{-1} \lambda} & \mapsto & \psi_{\mu}(w) u_{\lambda}
\end{eqnarray*}
Note that since $u_{\mu^{-1} \lambda} = \psi_{\mu}^{-1} ( u_{\lambda})$, this map is obtained by simply applying $\psi_{\mu}$ to the elements in $Au_{\mu^{-1} \lambda}$. It is $k$-linear, and for every element $v \in A$ and $w u_{\mu^{-1} \lambda} \in Au_{\mu^{-1} \lambda}$ there are equalities
\begin{eqnarray*}
\psi_{\mu} \left ( v \cdot ( w u_{\mu^{-1} \lambda} ) \right ) & = & \psi_{\mu} \left ( vw u_{\mu^{-1} \lambda} \right ) \\
& = & \psi_{\mu} (u) \psi_{\mu} (w) u_{\lambda} \\
& = & u \cdot \left ( \psi_{\mu} (w) u_{\lambda} \right )
\end{eqnarray*}
Thus the map is an $A$-homomorphism. Similarly, the inverse automorphism $\psi_{\mu}^{-1}$ induces an $A$-homomorphism in the other direction, hence $Au_{\mu^{-1} \lambda}$ and ${_{\psi_{\mu}}\left ( Au_{\lambda} \right )}$ are isomorphic $A$-modules. Using \cite[Proposition 3.5]{BerghErdmann} again, we now see that $\V_H \left ( {_{\psi_{\mu}}\left ( Au_{\lambda} \right )} \right )$ equals the line $\ell _{F( \mu^{-1} \lambda )}$.

Twisting an $A$-module $X$ by an automorphism $\psi$ is the same as tensoring with the bimodule ${_{\psi}A_1}$, i.e.\ ${_{\psi}X} \simeq {_{\psi}A_1} \otimes_A X$. Therefore, with $\lambda$ and $\mu$ as above, the support variety $\V_H \left ( {_{\psi_{\mu}}A_1} \otimes_A Au_{\lambda} \right )$ is the line $\ell _{F( \mu^{-1} \lambda )}$. On the other hand, the support variety $\V_H( Au_{\lambda} )$ is the line $\ell_{F(\lambda)}$, which generically differs from $\ell _{F( \mu^{-1} \lambda )}$. For example, with $\lambda = (1, \dots, 1)$, any $\mu$ whose components are not all the same when raised to the $a$th power will do. Consequently, for this $\lambda$ and such a $\mu$, we see that $\V_H \left ( {_{\psi_{\mu}}A_1} \otimes_A Au_{\lambda} \right ) \nsubseteq \V_H( Au_{\lambda} )$.
\end{proof}

As a consequence of the theorem, there cannot exist a bimodule version of the tensor product property for support varieties over the algebra $A^c_q$.

\begin{corollary}\label{cor:notintersection}
Let $H,M$ and $B$ be as in \emph{Theorem \ref{thm:main}}, and suppose that $\V_H^b$ is a support variety theory on the category of $A^c_q$-bimodules, defined in terms of the maximal ideal spectrum of $H$. Then $\V_H(B \otimes_{A^c_q} M) \neq \V_H^b (B) \cap \V_H(M)$.
\end{corollary}

For a finite dimensional algebra $A$, there are actually several possible ways of defining support varieties for bimodules. Namely, take any commutative graded subalgebra $H$ of $\HH^*( A )$. For a bimodule $B$, we can view $\Ext_{A^{\e}}^*(B,A)$ as a left module over $\HH^*( A )$, and in this way it becomes an $H$-module. We can then define
$$\V_H^b (B) = \left \{ \m \in \Maxspec H \mid \Ann_H \left ( \Ext_{A^{\e}}^*(B,A) \right ) \subseteq \m \right \}$$
Similarly, we can use the fact that $\Ext_{A^{\e}}^*(A,B)$ is a right module over $\HH^*( A )$ and obtain another support variety. These types of one-sided support varieties were studied in \cite{BerghSolberg}, where it was shown that they satisfy many of the properties one expects for a meaningful theory of support.

Now suppose that we take a bimodule $B$ which is projective as a left
$A$-module. Then if we take any exact sequence $\eta$ of bimodules,
the sequence $\eta \otimes_A B$ remains exact. Thus we obtain a ring
homomorphism
\begin{eqnarray*}
\HH^*(A) & \to & \Ext_{A^{\e}}^*(B,B) \\
\eta & \mapsto & \eta \otimes_A B
\end{eqnarray*}
of graded rings, and we can define
$$\V_H^b (B) = \left \{ \m \in \Maxspec H \mid \Ann_H \left (
    \Ext_{A^{\e}}^*(B,B) \right ) \subseteq \m \right \}$$ 
Similarly, if $B$ is projective as a right $A$-module, we obtain a
version by tensoring with $B$ on the left. Consequently, for bimodules
which are projective as both left and right $A$-modules, there are
totally at least four ways of defining support varieties using $H$,
and there is in general no reason to expect them to be equivalent.

Suppose now that $A$ is a finite dimensional selfinjective algebra satisfying \textbf{Fg} with respect to some subalgebra $H$ of its Hochschild cohomology ring. We then ask: what are the consequences of having a tensor product formula for bimodules acting on left modules? In order to investigate this, assume that
$$\V_H (B \otimes_A M) = \V_H^b (B) \cap \V_H(M)$$
for all $B$ in a tensor closed subcategory $\calX$ of bimodules and
all left $A$-modules $M$, where $\V_H$ is the usual support variety theory on left modules and $\V_H^b$ is some support variety theory for bimodules in
$\calX$ (defined in terms of the same geometric space as $\V_H$, namely the maximal ideal spectrum of $H$). Then
\begin{align}
\V_H^b(B_1 \otimes_A  B_2) \cap \V_H(M) & = \V_H( (B_1 \otimes_A B_2 ) \otimes_A M)\notag\\
                                       & = \V_H(B_1 \otimes_A (B_2 \otimes_A M))\notag\\
                                       & = \V_H^b(B_1)\cap \V_H(B_2 \otimes_A M)\notag\\
                                       & = \V_H^b(B_1)\cap \V_H^b(B_2)\cap \V_H(M)\notag\\
                                       & = \V_H^b(B_2)\cap \V_H^b(B_1)\cap \V_H(M)\notag\\
                                       & = \V_H(B_2 \otimes_A (B_1 \otimes_A M))\notag\\
                                       & = \V_H((B_2 \otimes_A B_1)\otimes_A M)\notag\\
                                       & = \V_H^b(B_2\otimes_A B_1)\cap \V_H(M)\notag                   
\end{align}
for all $B_1$ and $B_2$ in $\calX$ og all left $A$-modules $M$.
Then we claim that the equality 
$$\V_H^b(B_1 \otimes_A B_2) = \V_H^b(B_2 \otimes_A B_1)$$
holds for all bimodules $B_1$ and $B_2$ in $\calX$.  To see this, choose
$M = A/\ra$, where $\ra$ is the radical of $A$. Then $\V_H(M)$ is the whole defining maximal ideal spectrum of $H$, so
that $\V_H^b(B_1 \otimes_A B_2) = \V_H^b(B_2 \otimes_A B_1)$.  Hence, one consequence
is that the bimodule support variety $\V_H^b$ must be independent of the order of the
terms in a tensor product of bimodules, and therefore forcing some
type of symmetry on the tensor products of bimodules in $\calX$.

Let $\eta\colon \Omega_{A^{\e}}^n(A)\to A$ represent a homogeneous
element in $H$, where $\Omega_{A^{\e}}^n(A)$ is the $n$th syzygy in a
minimal projective resolution of $A$ over $A^{\e}$.  Taking the
pushout along this homomorphism and the minimal projective resolution
of $A$ over $A^{\e}$ gives rise to a short exact sequence
\[0\to A \to M_\eta\to \Omega_{A^{\e}}^{n-1}(A)\to
  0\]
as defined in \cite{EHSST}. The bimodules $M_{\eta}$ for homogeneous
elements $\eta$ in $H$ have the following property
\[ \V_H(M_{\eta_1}\otimes_A\cdots \otimes_A
  M_{\eta_t}\otimes_A M) = \V_H(\langle
  \eta_1,\ldots,\eta_t\rangle) \cap \V_H(M).\] 
If there is a support variety $\V_H^b$ of bimodules such that 
\[ \V_H^b(M_{\eta_1}\otimes_A\cdots \otimes_A
  M_{\eta_t}) =  \V(\langle \eta_1,\ldots,\eta_t\rangle),\]
then $\V_H^b$ must in particular satisfy 
\[ \V_H^b(M_{\eta_1}\otimes_A M_{\eta_2}) = 
\V_H^b(M_{\eta_2}\otimes_A M_{\eta_1}).\]

For example, let $\V_H^b(B) = \V_H(B \otimes_A A/\ra)$ for a bimodule $B$.  Then
is follows that  
\[ \V_H^b(M_{\eta_1}\otimes_A\cdots \otimes_A M_{\eta_t}) =
  \V_H(\langle\eta_1,\ldots,\eta_t\rangle)\]
for all homogeneous elements $\eta_i$ in $H$, and $\V_H^b$ satisfies the
above symmetry condition.  Since 
\begin{align}
\Ext^*_A \left ( B\otimes_A A/\ra, A/\ra \right ) & \simeq \Ext^*_{A^{\e}} \left ( B,
                                          \Hom_A(A/\ra, A/\ra) \right )\notag\\
& \simeq \Ext^*_{A^{\e}}(B, A/\ra \otimes_k A/\ra)\notag\\
& \simeq \Ext^*_{A^{\e}}(B, A^{\e}/\rad A^{\e})\notag
\end{align}
as $H$-modules, and $A/\ra \otimes_k A/\ra \simeq
A^{\e}/\rad A^{\e}$ when $A/\ra$ is separable over the field
$k$, then applying similar arguments as in \cite{SnashallSolberg} we
obtain that  
\begin{align}
\V_H^b(B) & = \V(\Ann_{H}\Ext^*_{A^{\e}}(B,A^{\e}/\rad
        A^{\e}))\notag\\
      & = \V(\Ann_{H}\Ext^*_{A^{\e}}(B,B))\notag\\ 
      & =
\V(\Ann_{H}\Ext^*_{A^{\e}}(A^{\e}/\rad A^{\e},B)).\notag
\end{align}
In other words, adapting the notion from \cite{SnashallSolberg}, 
\[ \V_H^b(B) = \V_H^b(B,A^{\e}/\rad A^{\e}) = \V_H^b(B,B) = \V_H^b(A^{\e}/\rad
  A^{\e},B).\] Then it is natural to ask how we can/should choose
$\calX$.  If we are thinking in terms of subcategories of the stable category
of bimodules, can we choose $\calX$ to be the tensor closed
subcategory generated by the bimodules $M_\eta$ for all homogeneous
elements $\eta$ in $H$? If all $M_\eta$'s are in $\calX$, we do not
know how $M_{\eta_1}\otimes_A M_{\eta_2}$ and
$M_{\eta_2}\otimes_A M_{\eta_1}$ are related as bimodules in general.

Let us now return to our quantum complete intersection $A^c_q$. Corollary \ref{cor:notintersection}, which is a direct consequence of Theorem \ref{thm:main}, shows that the tensor product property for support varieties over this algebra cannot hold in general, now matter how one defines support varieties for bimodules. Another consequence of Theorem \ref{thm:main} is that not all the
thick subcategories of the derived category and the stable module
category of $A^c_q$ are tensor ideals. In order to explain this, let
us first briefly describe a general framework where one typically is
interested in such questions; for details, we refer to
\cite{BuanKrauseSnashallSolberg}. Let $\C$ be a triangulated tensor
category, that is, a triangulated category which is at the same time a
(possibly non-symmetric) tensor category, and where the two structures
are compatible. Furthermore, suppose that $\C$ acts on a triangulated
category $\D$. This means that there exists an additive bifunctor
\begin{eqnarray*}
\C \times \D & \to & \D \\
( C,D ) & \mapsto & C \ast D
\end{eqnarray*}
which is compatible in a natural way with the structures of both $\C$
and $\D$. Finally, suppose that $H$ is a commutative graded subalgebra
of the graded endomorphism ring $\End^*_{\C}(I)$ of the unit object
$I$ in $\C$, or, more generally, that there exists a ring homomorphism
$H \to \End^*_{\C}(I)$. Then for all objects $D_1,D_2 \in \D$, the
graded homomorphism group $\Hom^*_{\D}(D_1,D_2)$ becomes a left and a
right $H$-module, and left and right scalar multiplication coincide up
to a sign. One can then define the support variety $\V_H(D_1,D_2)$ as
usual, in terms of the variety of the annihilator ideal
$\Ann_H \left ( \Hom^*_{\D}(D_1,D_2) \right )$. For a single object
$D \in \D$, one defines the support variety by $\V_H(D) = \V_H(D,D)$.

Given any triangulated category, it is of great interest to classify
its thick subcategories. The first example of such a classification
was the celebrated result of Hopkins-Neeman, for the category of
perfect complexes over a commutative noetherian ring (cf.\
\cite{Hopkins, Neeman}). That particular classification result showed
for free that all the thick subcategories are actually thick tensor
ideals. Now given $\C$ and $\D$ as above, one may ask for a similar
classification of thick subcategories of $\D$, and whether these are
all tensor ideals. Here, the notion of tensor ideals in $\D$ refers to
the action of $\C$ on $\D$: a thick subcategory $\A \subseteq \D$ is a
tensor ideal if $C \ast A \in \A$ for all $C \in \C$ and $A \in \A$.

Suppose that $V$ is a closed homogeneous subvariety of $\Maxspec H$,
and define a full subcategory $\A_V$ of $\D$ by
$$\A_V = \left \{ D \in \D \mid \V_H(D) \subseteq V \right \}$$
This is a thick subcategory of $\D$, and there are several classes of
examples of triangulated categories where \emph{all} the thick
subcategories are of this form. For example, this is the case for the
category of perfect complexes over a commutative noetherian ring. The
crucial point now is that whenever
$\V_H ( C \ast D ) \subseteq \V_H(D)$ for all objects $C \in \C$ and
$D \in \D$, then $\A_V$ is automatically a thick tensor ideal for all
$V$. This indicates the importance of the inclusion property
$$\V_H ( C \ast D ) \subseteq \V_H(D)$$
for support varieties in the setting of a triangulated tensor category
acting on a triangulated category.

Now consider our quantum complete intersection $A = A^c_q$ again. This
is a selfinjective algebra, and so the stable module category
$\stmod A$ is triangulated. The enveloping algebra $A^{\e}$ is also
selfinjective, and its stable module category $\stmod A^{\e}$, that
is, the stable module category of $A$-bimodules, is a triangulated
tensor category. It acts on $\stmod A$ by tensor products over $A$,
and so we are in a setting where all of the above applies. However,
let $H,M$ and $B$ be as in Theorem \ref{thm:main}. Since
$\V_H(B \otimes_A M) \nsubseteq \V_H(M)$, not all thick subcategories
of $\stmod A$ can be tensor ideals. Namely, take $V = \V_H(M)$ and
define $\A_V$ as above. This is a thick subcategory of $\stmod A$, but
it is not a tensor ideal since $M \in \A_V$ but
$B \otimes_A M \notin \A_V$. Finally, note that the bimodule $B$ we
used in the proof of Theorem \ref{thm:main} is actually projective as
a left and as a right $A$-module. The bounded derived category of such
bimodules is also a triangulated tensor category, and it acts on the
bounded derived category $\derived ( \mod A)$ of $A$-modules. Thus
also in $\derived ( \mod A)$ there are thick subcategories that are
not tensor ideals.

\end{document}